\pdfoutput=1
\documentclass[runningheads]{llncs}
\usepackage[T1]{fontenc}
\usepackage{graphicx}
\usepackage{amsmath}
\usepackage{mathtools}
\usepackage{amssymb}
\usepackage[hidelinks]{hyperref}
\usepackage{microtype}
\usepackage{subcaption}

\newcommand{\F}{\mathbb{F}}
\newcommand{\Z}{\mathbb{Z}}

\DeclareMathOperator{\RC}{RC\_CLASS}
\DeclareMathOperator{\ST}{ST\_CLASS}

\newcommand{\limp}{\mathbin\rightarrow}
\newcommand{\liff}{\mathbin\leftrightarrow}
\newcommand{\nliff}{\mathbin{\mspace{2mu}\not\mspace{-2mu}\leftrightarrow}}
\newcommand{\MOLS}[2]{#1\operatorname{MOLS}(#2)}

\begin{document}

\title{Orthogonal Latin Squares of Order Ten with Two Relations: A SAT Investigation}
\titlerunning{Orthogonal Latin Squares of Order Ten with Two Relations}

\author{Curtis Bright\inst{1,2,3}\orcidID{0000-0002-0462-625X} \and
Amadou Keita\inst{2}\orcidID{0009-0001-5861-4617} \and
Brett Stevens\inst{3}\orcidID{0000-0003-4336-1773}}
\authorrunning{C. Bright et al.}
\institute{School of Computer Science, University of Waterloo,
Canada\\
 \and
Department of Mathematics and Statistics, University of Windsor,
Canada \\
\and
School of Mathematics and Statistics, Carleton University,
Canada \\
\email{cbright@uwaterloo.ca}, \email{keitaa@uwindsor.ca}, \email{brett@math.carleton.ca}
}
\maketitle
\begin{abstract}
A $k$-net($n$) is a combinatorial design equivalent to $k-2$ mutually orthogonal Latin squares of order $n$.
A relation in a net is a linear dependency over $\F_2$ in the incidence matrix of the net.
A computational enumeration of all orthogonal pairs of Latin squares of order 10
whose corresponding nets have at least two nontrivial relations was achieved by Delisle in 2010
and verified by an independent search of Myrvold.
In this paper, we confirm the correctness of their exhaustive enumerations
with a satisfiability (SAT) solver approach instead of using custom-written backtracking code.
Performing the enumeration using a SAT solver has at least three advantages.
First, it reduces the amount of trust necessary, as SAT solvers produce independently-verifiable
certificates that their enumerations are complete.
These certificates can be checked by formal proof verifiers that are relatively simple pieces of software, and therefore easier to trust.
Second, it is typically more straightforward and less error-prone to use a SAT solver over writing search code.
Third, it can be more efficient to use a SAT-based approach, as SAT solvers are highly optimized
pieces of software incorporating backtracking-with-learning for improving the
efficiency of the backtracking search.  For example, the SAT solver completely enumerates
all orthogonal pairs of Latin squares of order ten with two nontrivial relations in under 2 hours on
a desktop machine, while Delisle's 2010 search used 11,700 CPU hours.
Although computer hardware was slower in 2010, this alone cannot explain the improvement
in the efficiency of our SAT-based search.
\keywords{Latin square \and orthogonal Latin square \and net \and satisfiability solving.}
\end{abstract}

\section{Introduction}

A $k$-net($n$) is a set of $n^2$ points and $kn$ lines with the following properties:

\begin{enumerate}
\item Every line contains $n$ points and every point lies on $k$ lines.
\item There are $k$ parallel classes of lines, with each parallel class containing $n$ lines
that do not intersect each other.
\item Every pair of lines from different parallel classes intersect exactly once.
\end{enumerate}
An \emph{incidence matrix} of a $k$-net($n$) is an $n^2\times kn$ matrix over $\F_2=\{0,1\}$ whose
$(i,j)$th entry is 1 exactly when the $i$th point lies on the $j$th line.
In this paper without loss of generality we assume that the lines in a net are always ordered
by parallel class, i.e., the first parallel class consists of the first
$n$ lines in the net, the second parallel class consists of the next $n$ lines in the net, etc.
Under this ordering, the axioms of a net imply that if $A$ is the incidence matrix of a $k$-net($n$), then
\begin{equation}
A^TA = \begin{bmatrix}
nI & J & \cdots & J \\
J & nI & \cdots & J \\
\vdots & \vdots & \ddots & \vdots \\
J & J & \cdots & nI
\end{bmatrix}
\end{equation}
where here $A$ and $A^T$ are considered as matrices over $\Z$,
$I$ is the identity matrix of order~$n$, and $J$ is the all-ones matrix of order~$n$.
This follows because the $(i,j)$th entry of $A^TA$ counts the number of points occurring
on both the $i$th and $j$th line, and this count is $1$ when the lines are in different parallel
classes, $0$ when the lines are distinct and in the same parallel class, and $n$ when the lines
are identical (i.e., $i=j$).

It is well-known that a $k$-net($n$) is equivalent to $k-2$ mutually orthogonal Latin squares of order $n$, denoted $\MOLS{(k-2)}{n}$,
and it is straightforward to convert a $k$-net to a collection of $k-2$ Latin squares and vice versa.
Roughly speaking, the net's first parallel class corresponds to Latin square rows, the net's second parallel class corresponds to
Latin square columns, and for $\ell>2$ the net's $\ell$th parallel class corresponds to symbols
of the $(\ell-2)$th Latin square.
When the $(i,j)$th entry of the $(\ell-2)$th Latin square contains symbol~$s$, its corresponding
net will contain a point that lies on the $i$th line of the first parallel class,
the $j$th line of the second parallel class, and the $s$th line
of the $\ell$th parallel class.

A \emph{relation} in a $k$-net($n$) is a linear dependency in the columns of $A$ over $\F_2$. The \emph{rank} of a $k$-net($n$) is the rank of its incidence matrix.  The rank of a $k$-net($n$) is at most $kn-k+1$ (c.f.~\cite{howard2010counterexample}) since there are $k-1$ trivial relations formed by the lines in the first parallel class and the $i$th parallel class for $2\leq i\leq k$.
Howard~\cite{howardthesis} considered nets of order $n\equiv 2\pmod{4}$ of which
the case $n=10$ is of particular interest.
She showed that a 6-net(10) has rank $\leq53$ and therefore contains at least two nontrivial relations, and also that a 4-net(10) has rank $\geq33$ and therefore contains at most four nontrivial relations.

Delisle, under the supervision of Myrvold, computationally enumerated all 4-nets(10) with at least two nontrivial relations~\cite{delisle2010search}.  They found that there exist no 4-nets(10) with four nontrivial relations, up to isomorphism there are six 4-nets(10) with exactly three nontrivial relations, and up to isomorphism there are 85 4-nets(10) with exactly two nontrivial relations.  More recently, Gill and Wanless~\cite{Gill2023} computationally enumerated all 4-nets(10) with at least one nontrivial relation and found that up to isomorphism there are exactly 18,526,229
4-nets(10) with exactly one nontrivial relation.

A 4-net(10) is equivalent to an orthogonal pair of Latin squares of order 10.  Latin squares of order 10 are of particular interest because 10 is the first order for which the largest collection of mutually orthogonal Latin squares is unknown.  It is known that 9 mutually orthogonal Latin squares of order 10 would be equivalent to a projective plane of order ten, but this was ruled out by exhaustive computer search~\cite{lam1989non}.  Combined with a result of Bruck~\cite{bruck1963finite}, this implies that $\MOLS{7}{10}$ do not exist.  Thus, the maximum number of mutually orthogonal Latin squares of order 10 is between 2 and 6.
As a consequence of the searches of Delisle~\cite{delisle2010search} and Gill and Wanless~\cite{Gill2023}, all $\MOLS{2}{10}$ with nontrivial relations are known and none of them are part of a $\MOLS{3}{10}$.

In this paper, we verify the searches of Delisle~\cite{delisle2010search} by exhaustively enumerating all $\MOLS{2}{10}$ with two nontrivial relations.  Our approach differs from Delisle's original search and the independent verification of Gill and Wanless~\cite{Gill2023} because our search uses a Boolean satisfiability (SAT) solver.  We reduce the existence of a $\MOLS{2}{10}$ with two relations into a problem of Boolean logic and then use a SAT solver to find all solutions of the logic problem (and correspondingly all $\MOLS{2}{10}$ with two relations).
SAT solvers can be surprisingly effective at solving various problems in mathematics~\cite{bright2020effective} and recently they
have been used with increasing frequency to solve problems in combinatorics and design theory~\cite{Zhang2021}.  For example, in 2021 they were used to verify Lam et al.'s result that order ten projective planes do not exist~\cite{bright2021sat}.

Traditionally, backtracking algorithms are used to computationally search for combinatorial designs~\cite{kaski2006}.  SAT solvers offer an alternative approach to the backtracking paradigm.  Although SAT solvers also perform a form of backtracking, they also use a powerful learning process known as conflict-driven clause-learning~\cite{MarquesSilva1996}.  This technique along with many other optimizations and heuristics that have been fine-tuned over decades enables SAT solvers to outperform traditional backtracking search in many problems of interest.

Moreover, even if SAT solvers were not as efficient as custom backtracking approaches, they have the advantage of making the search process less error-prone because using a SAT solver does not require writing any code for performing a search.  Unfortunately, it is a reality of software development that almost all computer programs have bugs, and writing efficient search code is an inherently error-prone process~\cite{Lam1990}.  This is a particularly important consideration when a computer program purports to perform an exhaustive classification of a mathematical object.  How can we trust that a bug in the program did not cause some objects to be missed in the search?  It is typically impossible for even professional programmers to guarantee their code has no bugs due to the inherent difficulties in writing computer code.  One way of decreasing the chance that a bug results in a missed object is to perform the same search with two implementations written independently.  While this does reduce the chance of something being missed, it still relies on trusting code that cannot be certified to be correct.  Indeed, it has happened that mathematical designs have been missed by a search with an independent verification, for example, in Lam's problem~\cite{Bright2022} and the enumeration of good matrices of order 27~\cite{Bright2019}.

Using a SAT solver sidesteps the need for writing search code.  Instead, one generates a ``SAT encoding'' specifying the properties that the object in question must satisfy, and then the SAT solver searches for the object.  Moreover, the SAT solver \emph{does not need to be trusted itself}, because during the search it produces a certificate that can be checked for correctness by a proof verifier---a simpler piece of software that can be written independently from the SAT solver.  This approach does rely on the SAT encoding being correct, but this typically requires less trust than would be required of a search procedure.
We describe the SAT encoding we use in Section~\ref{sec:encoding}, a crucial component of which is a SAT encoding of the symmetry breaking used in Delisle's original search
(described in detail along with other background in Section~\ref{sec:background}).

SAT solvers do not always perform well on mathematical problems, particularly
when there is underlying structure in the problem unknown to the solver.  As an example,
integer factorization can be reduced to the SAT problem, but
SAT solvers cannot factorize integers nearly as efficiently
as algorithms specifically designed to do so using number theoretic concepts.
Augmenting SAT solvers with number theoretic reasoning improves their performance,
but even still they only outperform traditional algorithms when extra information is available
such as random bits of the prime factors~\cite{Ajani2024}.
Despite SAT solvers having no efficiency guarantee,
they did perform remarkably well in the enumeration of orthogonal pairs of Latin squares of order ten with two relations.
Only minor modifications to the SAT solver and proof verifier were required, as described in Section~\ref{sec:enumerating}.
We enumerated all 91 pairs of Latin squares of order ten with two relations in less than 1.75 hours on a desktop computer (see Section~\ref{sec:results}).
Conversely, Delisle's original backtracking search, written in the programming language C and optimized for speed, used around 488 CPU days in 2010.
While some of the improvements in speed of our SAT-based search is undoubtedly from the improvements in computer hardware,
this alone cannot account for an over $6000\times$ speedup in CPU time, demonstrating the computational efficiency of SAT solvers on problems of this type.

\section{Background}\label{sec:background}

The \emph{type} of a relation is a list of the number of lines each parallel class contributes
to the relation.  For example, a relation in a 4-net($n$) that has 4 lines from the first two
parallel classes and 6 lines from the last two parallel classes has type $[4,4,6,6]$.
Because the union of two parallel classes form a trivial relation, without loss of generality
any relation can be ``complemented'' and written in a form where at most one entry
in the type is larger than $n/2$.
In her PhD thesis, Howard~\cite{howardthesis} studies relations in nets of order $n\equiv 2\pmod{4}$.
In particular, she proved the following proposition.

\begin{proposition}[{\cite[Prop.~5.4]{howardthesis}}]\label{prop4netrel}
Every nonempty relation in a 4-net of order $n\equiv2\pmod4$ with at most $n/2$ lines
in three classes must be of the type
$[k,k,k,k]$ where $k$ is an even integer with $n/3\leq k<n/2$.
\end{proposition}

For example, Proposition~\ref{prop4netrel} implies that every nontrivial relation in
a 4-net(10) is complementable to one of type $[4,4,4,4]$, because every nontrivial relation
in a 4-net(10) is complementable to a nonempty relation with at most 5 lines in three classes.
Now consider the case of a 4-net(10) with two linearly independent nontrivial relations.
Note the sum of two nontrivial relations $R_1$
and~$R_2$ will be a third relation, and because
the relations are over $\F_2$, the sum is the symmetric difference $(R_1\cup R_2)\setminus(R_1\cap R_2)$.  Although this third relation
will not be linearly independent to $R_1$ and $R_2$, it will be nontrivial, and therefore by Proposition~\ref{prop4netrel}
complementable to a relation of the type $[4,4,4,4]$.
The following proposition constrains the manner in which two linearly independent relations intersect.

\begin{proposition}[cf.~\cite{delisle2010search}]\label{prop2reltypes}
Suppose\/ $R_1$ and\/ $R_2$ are two linearly independent relations of type $[4,4,4,4]$ in a 4-net(10).
Then each parallel class contains exactly one or two lines from both\/ $R_1$ and\/ $R_2$.
\end{proposition}
\begin{proof}
Suppose $x$ is the number of lines in the first parallel class and both $R_1$ and $R_2$.
Since the sum of $R_1$ and $R_2$ is a third relation that is complementable to one of type $[4,4,4,4]$,
we have that $(4-x)+(4-x)$, the number of lines in the third relation and the first parallel class,
is either $4$ or $10-4=6$.
In the former case $x=2$ and in the latter case $x=1$.  The same argument applies to every parallel class. \qed
\end{proof}

Following Delisle, we will suppose the lines appearing in both $R_1$ and $R_2$ are ordered
to appear first in each parallel class, followed by the lines in $R_1$ and not $R_2$, and then
the lines in $R_2$ and not $R_1$.  The remaining lines, those not in either $R_1$ and $R_2$, appear last.
Say that there are $x_i$ lines in parallel class $i$ and $R_1\cap R_2$, there are
$y_i$ lines in parallel class~$i$ and $R_1\setminus R_2$, and there are $z_i$ lines in parallel class~$i$
and $R_2\setminus R_1$.  We use the notation $[[x_1,y_1,z_1],[x_2,y_2,z_2],[x_3,y_3,z_3],[x_4,y_4,z_4]]$
to denote the form of $R_1$ and $R_2$.  By Proposition~\ref{prop2reltypes}, without loss of generality
the possible forms can be taken to be one of the five cases
\begin{gather}
[[1,3,3],[1,3,3],[1,3,3],[1,3,3]], \tag{case 1} \\
[[1,3,3],[1,3,3],[1,3,3],[2,2,2]], \tag{case 2} \\
[[1,3,3],[1,3,3],[2,2,2],[2,2,2]], \tag{case 3} \\
[[1,3,3],[2,2,2],[2,2,2],[2,2,2]], \tag{case 4} \\
[[2,2,2],[2,2,2],[2,2,2],[2,2,2]]. \tag{case 5}
\end{gather}
Furthermore, Delisle used a counting argument to rule out
cases~2 and~4~\cite[pg.~17]{delisle2010search}.  Thus, Delisle's search focused on cases 1, 3, and~5.
Interestingly, using our approach a SAT solver rules out cases 2, 3, and~4 in a few seconds each.
Given case~3 was ruled out by Delisle using 23 days of compute time, the SAT solver
shows its strength at uncovering contradictions by ruling out case~3 quite quickly.
At the end of Section~\ref{sec:results}, we also give an argument using no search and a short computation
that rules out case~3.  We found this argument
independently of the SAT computation.

\subsection{Delisle's symmetry breaking}\label{sec:symmetry-breaking}

In order to perform an exhaustive enumeration up to isomorphism it is advantageous to restrict the search
space as much as possible without losing any solutions up to isomorphism---this process is known as
\emph{symmetry breaking}.  In this section we recount the symmetry breaking used in Delisle's
thesis~\cite{delisle2010search} for 4-nets(10) with two relations (stated in terms of
$\MOLS{2}{10}$)\@.  Say $(A,B)$ is a $\MOLS{2}{10}$ whose corresponding net has 2 nontrivial relations.
Delisle's method for adding symmetry breaking constraints on $(A,B)$ is based on adding additional
constraints on the entries in the first column and row of $A$ and $B$.

Suppose $(A,B)$ is a pair of orthogonal Latin squares with two relations of the form
\begin{equation}
[[x_1,y_1,z_1],[x_2,y_2,z_2],[x_3,y_3,z_3],[x_4,y_4,z_4]] .
\end{equation}
The form of the relations defines equivalence classes on the rows (from parallel class 1), columns
(from parallel class 2), symbols of $A$ (from parallel class 3), and symbols of~$B$ (from parallel class 4).
For example, the row equivalence classes are determined by the values of $x_1$, $y_1$, and $z_1$:
explicitly, the equivalence classes of rows will be defined by the index sets $[0,x_1)$,
$[x_1,x_1+y_1)$, $[x_1+y_1,x_1+y_1+z_1)$, and $[x_1+y_1+z_1,10)$.
Note that the following six equivalence operations
on pairs of Latin squares $(A,B)$ preserve
the orthogonality of the squares and the form of the relations.
\begin{enumerate}
\item Permutation of rows of $A$ and $B$ preserving row equivalence classes (the same permutation applied to $A$ and $B$ simultaneously).
\item Permutation of columns of $A$ and $B$ preserving column equivalence classes (the same permutation applied to $A$ and $B$ simultaneously).
\item Permutation of the symbols of $A$ preserving $A$-symbol equivalence classes.
\item Permutation of the symbols of $B$ preserving $B$-symbol equivalence classes.
\item Taking the transpose of $A$ and $B$ (when the row and column equivalence classes match, i.e., cases 1--3 and 5).
\item Swapping $A$ and $B$ (when the $A$-symbol and $B$-symbol equivalence classes match, i.e., cases 1 and 3--5).
\end{enumerate}

Fix an ordering of the entries of a Latin square in the following way:
the entries of the first column (from top to bottom) are first, and then the entries of the
first row (from left to right) are next.  The remaining entries can be ordered arbitrarily.
In this way, if $A$ and $A'$ are two distinct Latin squares
we say $A<A'$ if on the first entry in which $A$ and $A'$ differ, say at
index $(i,j)$, we have $A_{ij}<A'_{ij}$.
Similarly, pairs of Latin squares can be ordered
after providing an ordering on pairs of symbols.
For this, a lexicographic ordering is used: say that $(a,b)<(a',b')$ when either $a<a'$ or $a=a'$ and $b<b'$.
Then, if $(A,B)$ and $(A',B')$ are two distinct pairs of Latin squares, we say
$(A,B)<(A',B')$ if on the first entry in which $(A,B)$ and $(A',B')$
differ, say at index $(i,j)$, we have $(A_{ij},B_{ij})<(A'_{ij},B'_{ij})$.

A pair of orthogonal Latin squares $(A,B)$ is said to
be a \emph{minimal pair} if $(A,B)$
cannot be decreased under the ordering described above by using the
equivalence operations described above.
Delisle's symmetry breaking method is based on the following six propositions.
In each, $(A,B)$ is a pair of orthogonal Latin squares and each proposition
gives a necessary condition for $(A,B)$ to be a minimal pair.

\begin{proposition}\label{prop:min-trans}
If $(A,B)$ is a minimal pair then $(A_{1,0},B_{1,0})<(A_{0,1},B_{0,1})$ (except possibly in case 4).
\end{proposition}

\begin{proof}
Suppose $(A,B)$ is a minimal pair with $(A_{1,0},B_{1,0})\geq(A_{0,1},B_{0,1})$.
Since $A$ and $B$ are orthogonal, $(A_{1,0},B_{1,0})$ and $(A_{0,1},B_{0,1})$ are distinct,
and thus $(A_{1,0},B_{1,0})>(A_{0,1},B_{0,1})$.  Applying the transpose operation to
$A$ and $B$ does not affect $A_{0,0}$ and $B_{0,0}$, but replaces $(A_{1,0},B_{1,0})$
with $(A_{0,1},B_{0,1})$, so $(A^T,B^T)<(A,B)$ in contradiction to the fact that
$(A,B)$ is minimal.  (This argument does not work in
case 4, as the transpose operation is not an equivalence operation in case 4.) \qed
\end{proof}

\begin{proposition}\label{prop:min-swap}
If $(A,B)$ is a minimal pair then $A<B$ (except possibly in case~2).
\end{proposition}

\begin{proof}
Suppose $(A,B)$ is a minimal pair with $A\geq B$.  Since $A$ and $B$ are orthogonal,
$A\neq B$, and thus $A>B$ and there is some entry on which $A$ and $B$ do not match.
After swapping $A$ and $B$ the first entry
on which the mismatch occurs will still be in the same place, so $(B,A)<(A,B)$
in contradiction to $(A,B)$ being minimal.
(This argument does not work in case 2, as swapping $A$ and $B$ is not an equivalence
operation in case~2.) \qed
\end{proof}

\begin{proposition}\label{prop:sort-col}
If $(A,B)$ is a minimal pair then the
symbols in the first column of\/ $A$ appear in sorted order within the rows of each row equivalence class.
\end{proposition}

\begin{proof}
Suppose $(A,B)$ is a minimal pair with the symbols in the first column of~$A$
not in sorted order within the rows of each row equivalence class.  Using row permutations,
permute the rows of $(A,B)$ to form $(A',B')$ such that the rows of the first column of $A'$
are now sorted within the rows of each row equivalence class.
This is possible because we can freely permute rows of $A$ and~$B$ if we focus
only on the rows in one particular row equivalence class.  Thus, there is
some permutation which sorts those particular rows by the entries of the first column of $A$.
Consider the first entry $A'_{i,0}$ of $A'$ that has changed after applying these permutations.
(This will also be the first entry of $B'$ that has changed, since the same permutations
are applied to $A$ and $B$.)  Because $A'_{i,0}<A_{i,0}$ and $(A_{j,0}',B'_{j,0})=(A_{j,0},B_{j,0})$ for all $j<i$ we have $(A',B')<(A,B)$
in contradiction to $(A,B)$ being minimal. \qed
\end{proof}

\begin{proposition}\label{prop:sort-row}
If $(A,B)$ is a minimal pair then the
symbols in the first row of\/ $A$ appear in sorted order within the columns of each column equivalence class.
\end{proposition}

\begin{proof}
Suppose $(A,B)$ is a minimal pair with the symbols in the first row of $A$
not in sorted order within the columns of each column equivalence class.
First consider cases 1--3, in which case the first column equivalence class
consists of only the first column.  In this case, the proposition is vacuous for the
first column equivalence class, as the list $[A_{0,0}]$ has length 1 and
is vacuously sorted.  Using column permutations, permute the columns of $(A,B)$
to form $(A',B')$ such that the columns of the first row of $A'$ are now
sorted within the columns of the remaining three column equivalence classes.
Note that the first column of $(A',B')$ matches the first column of $(A,B)$
since the first column was not permuted.  Consider the first entry $A'_{0,i}$
that has changed after applying these permutations.  Because $A'_{0,i}<A_{0,i}$
and $(0,i)$ is the first entry in which $(A',B')$ differs from $(A,B)$, it follows
that $(A',B')<(A,B)$, in contradiction to $(A,B)$ being minimal.

In cases 4 and 5, the above argument works to sort the entries in the first row of
$A$ in each of the last three column equivalence classes, but not the first, so
$A_{0,0}$ and $A_{0,1}$ may be out of order.
If $A_{0,0}>A_{0,1}$, swap the first two columns of $(A,B)$ to form $(A',B')$.
Since $A'_{0,0}=A_{0,1}<A_{0,0}$, we have $(A',B')<(A,B)$, in contradiction to
$(A,B)$ being minimal. \qed
\end{proof}

\begin{proposition}\label{prop:min-a-sym}
If $(A,B)$ is a minimal pair then for each $A$-symbol equivalence class, the
symbols of that equivalence class in the first column of\/ $A$ appear in sorted order.
\end{proposition}

\begin{proof}
Suppose $(A,B)$ is a minimal pair where the symbols in the same $A$-symbol equivalence
class in the first column of $A$ do not appear in sorted order.
Permute the symbols of $A$ to form $A'$ so that for each $A$-symbol equivalence class
the symbols in that equivalence class in the first column of $A'$ appear in sorted order.
Consider the first entry $A'_{i,0}$ that differs from $A_{i,0}$.  Since $A'_{i,0}<A_{i,0}$
and $A'_{j,0}=A_{j,0}$ for all $j<i$, we have $(A',B)<(A,B)$ in contradiction to $(A,B)$
being minimal. \qed
\end{proof}

\begin{proposition}\label{prop:min-b-sym}
If $(A,B)$ is a minimal pair then for each $B$-symbol equivalence class, the
symbols of that equivalence class in the first column of\/ $B$ appear in sorted order.
\end{proposition}

\begin{proof}
Suppose $(A,B)$ is a minimal pair where the symbols in the same $B$-symbol equivalence
class in the first column of $B$ do not appear in sorted order.
Permute the symbols of $B$ to form $B'$ so that for each $B$-symbol equivalence class
the symbols in that equivalence class in the first column of $B'$ appear in sorted order.
Consider the first entry $B'_{i,0}$ that differs from $B_{i,0}$.  Since $B'_{i,0}<B_{i,0}$
and $B'_{j,0}=B_{j,0}$ for all $j<i$, we have $(A,B')<(A,B)$ in contradiction to $(A,B)$
being minimal. \qed
\end{proof}

Finally, we prove another property of minimal pairs that we exploit in our encoding.

\begin{proposition}\label{prop:a00b00}
If $(A,B)$ is a minimal pair then $A_{0,0}=B_{0,0}$ (except possibly in case 2).
\end{proposition}

\begin{proof}
By Proposition~\ref{prop:min-a-sym} if $(A,B)$ is a minimal pair then $A_{0,0}$
must be in $\{0,1,4,7\}$ in case 1 and in $\{0,2,4,6\}$ in cases 3--5.  Similarly, by
Proposition~\ref{prop:min-b-sym} if $(A,B)$ is a minimal pair then $B_{0,0}$ must
be in $\{0,1,4,7\}$ in case 1 and in $\{0,2,4,6\}$ in cases~3--5.
Delisle~\cite[pg.~18]{delisle2010search} gives the possibilities
for the values of $(A_{0,0},B_{0,0})$ in case~1, and the only ones
which are in $\{0,1,4,7\}\times\{0,1,4,7\}$ are $(0,0)$, $(1,1)$, $(4,4)$, and $(7,7)$.
Similarly, \cite[pg.~20]{delisle2010search} gives the possibilities for the values of $(A_{0,0},B_{0,0})$ in case~3
(and these are identical in cases~4 and~5),
and the only ones which are in $\{0,2,4,6\}\times\{0,2,4,6\}$ are
$(0,0)$, $(2,2)$, $(4,4)$, and $(6,6)$. \qed
\end{proof}

\section{SAT Encoding}\label{sec:encoding}

In this section we describe our SAT encoding for the problem
of enumerating $\MOLS{2}{10}$ with two nontrivial relations.
In order to encode a pair of orthogonal Latin squares $(A,B)$
of order $n$ we use the $2n^3$ Boolean variables $A_{ijk}$ and
$B_{ijk}$ for $0\leq i<n$.  The variable $A_{ijk}$ will be true exactly when
the $(i,j)$th entry of square $A$ contains the symbol~$k$, and
similarly for the variables $B_{ijk}$ and the entries of the square $B$.

Modern SAT solvers require the input formulae to be in a format
known as conjunctive normal form.  An expression in Boolean logic
is in \emph{conjunctive normal form} when it is a conjunction of clauses,
a \emph{clause} being a disjunction of variables or negated variables.
For example, $\lnot x\lor y\lor z$ is a clause.
We may use the implication operator to express clauses
with the meaning that $(x_1\land\dotsb\land x_n)\limp(y_1\lor\dotsb\lor y_m)$
is shorthand for the clause $\lnot x_1\lor\dotsb\lor\lnot x_n\lor y_1\lor\dotsb\lor y_m$.

Our SAT encoding contains four kinds of constraints:
constraints asserting that $A$ and $B$ are Latin squares (see Section~\ref{sec:latin-enc}),
constraints asserting that $A$ and $B$ are orthogonal (see Section~\ref{sec:ortho-enc}),
constraints asserting that $A$ and $B$ have two nontrivial relations
and are in one of the forms specified by cases 1--5 (see Section~\ref{sec:relation-enc}),
and finally symmetry breaking constraints asserting that $A$ and $B$ satisfy
Propositions~\ref{prop:min-trans} to~\ref{prop:min-b-sym} (see Section~\ref{sec:symbreak-enc}).

\subsection{Latin square encoding}\label{sec:latin-enc}

Considering the Boolean variables $A_{ijk}$ as integers ($0$ for false
and $1$ for true), in order to specify that they encode a Latin square of order 10 we
need to enforce the following three constraints.
\begin{enumerate}
\item $\sum_{k=0}^{9} A_{ijk}=1$ for all $0\leq i,j\leq 9$ (every cell has exactly one symbol).
\item $\sum_{j=0}^{9} A_{ijk}=1$ for all $0\leq i,k\leq 9$ (every row contains every symbol exactly once).
\item $\sum_{i=0}^{9} A_{ijk}=1$ for all $0\leq j,k\leq 9$ (every column contains every symbol exactly once).
\end{enumerate}
The most straightforward
way of representing the constraint $\sum_{i=1}^{n}x_i=1$ in Boolean logic is to encode
$\sum_{i=1}^{n}x_i\leq 1$ via $\bigwedge_{i<j}(\lnot x_i\lor\lnot x_j)$
and $\sum_{i=1}^{n}x_i\geq 1$ via $\bigvee_{i=1}^{n} x_i$.
This encoding performed well in our experiments, but
we observed slightly better performance using Sinz's sequential counter encoding~\cite{sinz2005towards}.

To encode $\sum_{i=1}^n x_i\leq 1$ in the sequential counter encoding, first
the new variables $s_1$, $\dotsc$, $s_n$ are introduced ($s_i$ representing
that at least one of $x_1$, $\dotsc$, $x_i$ are true) using the $2n-1$ clauses
\begin{equation} x_i\limp s_i \quad\text{and}\quad s_{i-1}\limp s_i \end{equation}
for $1\leq i\leq n$ (when $i=1$ the clause $s_0\limp s_1$ is skipped).
Once the $s_i$ variables have been introduced, $\sum_{i=1}^n x_i\leq 1$ is encoded
using the $n-1$ additional clauses $\lnot x_i\lor\lnot s_{i-1}$ for $2\leq i\leq n$,
the idea being that $x_i$ and $s_{i-1}$ can never both be true, because that would imply
at least two variables in $x_1$, $\dotsc$, $x_n$ are true.
Moreover, $\sum_{i=1}^n x_i\geq 1$ can be encoded by setting $s_n$ to true and
adding the clauses $s_i\limp(s_{i-1}\lor x_i)$ for $1\leq i\leq n$ (when $i=1$
the literal $s_0$ is left out).
Setting $s_n$ to true causes two clauses to be trivially satisfied, so they can be removed.
Altogether, we encode $\sum_{i=1}^n x_i=1$ using $4n-4$ clauses.

\subsection{Orthogonality encoding}\label{sec:ortho-enc}

The orthogonality
of two Latin squares $A$ and $B$ of order~$n$ can be specified by the logical constraints
\begin{equation} (A_{ij}=k\land A_{i'j'}=k\land B_{ij}=l\land B_{i'j'}=l)\limp(i=i') \end{equation}
for $0\leq i,j,i',j',k,l<n$ (cf.~Zhang~\cite[Lemma~1]{zhang1997specifying}).
Taking the contrapositive and writing
this using the variables $A_{ijk}$ and $B_{ijk}$, this is equivalent
to the clauses
$\lnot A_{ijk}\lor\lnot A_{i'j'k}\lor\lnot B_{ijl}\lor\lnot B_{i'j'l}$
where $0\leq i,j,i',j',k,l<n$ with $i\neq i'$.  This encoding
of orthogonality uses $O(n^6)$ clauses of length 4.  An alternative
encoding of orthogonality that performs better in practice~\cite[Lemma 2]{zhang1997specifying}
uses $O(n^4)$ clauses of length 3 and $n^3$ new auxiliary variables.  To describe
this orthogonality encoding, we follow the derivation of Bright, Keita, and Stevens~\cite{myrvoldmols}
based on a composition square.

Consider the rows of a Latin square $X$ of order~$n$ as a collection of $n$ permutations of
the symbols $\{0,\dotsc,n-1\}$.  The row inverse square $X^{-1}$ is defined to be the Latin square whose rows are formed
by the inverses of the rows of $X$, and the composition square $XY$ is defined to be the square whose $i$th
row is the $i$th row of $X$ composed with the $i$th row of $Y$ (in a right-to-left way).
Note that the square $XY$ is \emph{not} a Latin square in general.
We now provide a theorem that the orthogonality encoding we use relies on.
\begin{theorem}[Mann~\cite{mann1942construction}]
Two Latin squares $A$ and $B$ are orthogonal if and only if\/ $AB^{-1}$ is a Latin square.
\end{theorem}
Let $Z$ denote the composition square $AB^{-1}$, and let the Boolean variables
$Z_{ijk}$ be true exactly when the $(i,j)$ entry of $Z$ contains the symbol~$k$
(where $0\leq i,j,k<n$).  The square $Z$ can be specified to be a Latin square using the encoding
from Section~\ref{sec:latin-enc}\@.  In order to encode $Z=AB^{-1}$, we need to enforce
that the $(i,B_{ij})$th entry of $Z$ contains the symbol $A_{ij}$.  This is done using the clauses
\begin{equation} (A_{ijk}\land B_{ijl})\limp Z_{ilk} \end{equation}
for $0\leq i,j,k,l<n$.  In fact, from $A=ZB$ we also derive the similar clause
$(Z_{ilk}\land B_{ijl})\limp A_{ijk}$, and from $B=Z^{-1}A$ we derive
$(Z_{ilk}\land A_{ijk})\limp B_{ijl}$.
These last two types of clauses are logically redundant, but in practice
they improve the performance of the SAT solver as they allow the solver to make
additional useful propagations.

\subsection{Relation encoding}\label{sec:relation-enc}

Let $R_1$ denote the indices
of the lines in the first relation, and let $R_2$ denote the indices
of the lines in the second relation.
Delisle~\cite{delisle2010search} defines an equivalence class on $(\text{row}, \text{column})$  Latin square index pairs
using a labelling function called $\RC$ that we describe below.
In the following, $i$ represents
a Latin square row index, and~$j$ represents a Latin square column index, so $0\leq i,j\leq9$.
Recall that we order the lines of a 4-net(10) so that lines 0 to 9 correspond to row indices of Latin squares
while lines 10 to 19 correspond to column indices of Latin squares.
Under such an ordering, note that line $j+10$ corresponds to the $j$th column of the Latin squares.
In what follows
the notation $x\liff y$ means $x$ and $y$ have the same truth value (both
true or both false), while $x\nliff y$ means $x$ and $y$ take opposite truth values.
Delisle's $\RC(i,j)$ labelling function is now defined by
\begin{equation}
\begin{cases}
0 & \text{if ($i\in R_1\liff j+10\in R_1$) and ($i\in R_2\liff j+10\in R_2$)}, \\
1 & \text{if ($i\in R_1\liff j+10\in R_1$) and ($i\in R_2\nliff j+10\in R_2$)}, \\
2 & \text{if ($i\in R_1\nliff j+10\in R_1$) and ($i\in R_2\liff j+10\in R_2$)}, \\
3 & \text{if ($i\in R_1\nliff j+10\in R_1$) and ($i\in R_2\nliff j+10\in R_2$)}.
\end{cases}
\end{equation}
Similarly, Delisle defines a labelling function $\ST$ on symbol pairs
$(s,t)$ where $s$ is a symbol of the first Latin square $A$ and $t$ is a symbol of the second Latin
square~$B$.
Representing $s$ and $t$ as integers in $\{0,\dotsc,9\}$,
note that line $s+20$ of the net corresponds to symbol~$s$ in the first Latin square,
and line $t+30$ of the net corresponds to symbol~$t$ in the second Latin square.
Concretely, $\ST(s,t)$ is defined
to be
\begin{equation}
\begin{cases}
0 & \text{if ($s+20\in R_1\liff t+30\in R_1$) and ($s+20\in R_2\liff t+30\in R_2$)}, \\
1 & \text{if ($s+20\in R_1\liff t+30\in R_1$) and ($s+20\in R_2\nliff t+30\in R_2$)}, \\
2 & \text{if ($s+20\in R_1\nliff t+30\in R_1$) and ($s+20\in R_2\liff t+30\in R_2$)}, \\
3 & \text{if ($s+20\in R_1\nliff t+30\in R_1$) and ($s+20\in R_2\nliff t+30\in R_2$)}.
\end{cases}
\end{equation}
Delisle then notes that the condition that relations $R_1$ and $R_2$ exist in the 4-net(10)
corresponding to the orthogonal Latin square pair $(A,B)$
is equivalent to the condition
\begin{equation} \RC(i,j)=\ST(A_{ij},B_{ij}) \text{ for all $0\leq i,j\leq9$}. \end{equation}
That is, if $(A_{ij},B_{ij})=(s,t)$ then $\RC(i,j)=\ST(s,t)$.
We encode this condition directly into our SAT encoding.
Each case~1--5 is encoded separately, so the forms of $R_1$
and $R_2$ are known, meaning that the values of $\RC(i,j)$ and $\ST(s,t)$
can be determined in advance for all $0\leq i,j,s,t\leq9$.
We encode the relation constraint in contrapositive form:
for all $i$, $j$, $s$, $t$ with $\ST(s,t)\neq\RC(i,j)$
we enforce that $(A_{ij},B_{ij})\neq(s,t)$, i.e., $A_{ij}\neq s$
or $B_{ij}\neq t$.  In Boolean logic, this becomes the clauses
\begin{equation} \lnot A_{ijs}\lor\lnot B_{ijt} \text{ where $\RC(i,j)\neq\ST(s,t)$} \end{equation}
for all $0\leq i,j,s,t\leq 9$.

\subsection{Symmetry breaking}\label{sec:symbreak-enc}

In this section we describe how we encode Delisle's symmetry breaking
propositions into conjunctive normal form.
In particular, we encode properties of minimal pairs of orthogonal
Latin squares with two relations described in Section~\ref{sec:symmetry-breaking}, as such
constraints do not remove any solutions up to isomorphism.  Thus, in this section
we assume that $(A,B)$ is a minimal pair of orthogonal Latin squares.

First, consider Proposition~\ref{prop:min-trans}, which applies
in all cases except case~4.  It says that
$(A_{1,0},B_{1,0})<(A_{0,1},B_{0,1})$.  First, we encode the weaker constraint
that $A_{1,0}\leq A_{0,1}$ in conjunctive normal form via
\begin{equation} \bigwedge_{\substack{0\leq k,l\leq9\\k>l}}(A_{1,0,k}\limp\lnot A_{0,1,l}) . \end{equation}
Next, we encode that when $A_{1,0}=A_{0,1}$ we have $B_{1,0}<B_{0,1}$.  This
is done via
\begin{equation} \bigwedge_{\substack{0\leq k,l,m\leq9\\ k\geq l}}\bigl((A_{1,0,m}\land A_{0,1,m}\land B_{1,0,k})\limp \lnot B_{0,1,l}\bigr) . \end{equation}

Now consider Proposition~\ref{prop:sort-col}, which says that the symbols in the
first column of~$A$ appear in sorted order within the rows in the same row
equivalence class.
Let~$R$ denote a row equivalence class, and let $R'\coloneqq R\setminus\{\max(R)\}$.  For example,
in cases~1--4, the possible nonempty values for $R'$ are $\{1,2\}$, $\{4,5\}$, and $\{7,8\}$.
In case~5, the possible values for $R'$ are $\{0\}$, $\{2\}$, $\{4\}$, and $\{6,7,8\}$.
In order to ensure that the symbols in the first column of $A$ with rows in $R$ appear in sorted order we
use the constraints
\begin{equation} \bigwedge_{\substack{i\in R'\\0\leq l<k\leq9}}(A_{i,0,k}\limp\lnot A_{i+1,0,l}) . \end{equation}
Proposition~\ref{prop:sort-row} says that the symbols in the first row of $A$ appear in
sorted order within the columns in the same column equivalence class and can be handled similarly.
Let $C$ denote a column equivalence class and let $C'\coloneqq C\setminus\{\max(C)\}$.  We ensure
the symbols in the first row of $A$ with columns in $C$ appear in sorted order using the constraints
\begin{equation} \bigwedge_{\substack{j\in C'\\0\leq l<k\leq9}}(A_{0,j,k}\limp\lnot A_{0,j+1,l}) . \end{equation}

Proposition~\ref{prop:min-a-sym} says that the symbols in the same $A$-symbol equivalence
class are sorted in the first column of $A$.  Let $S$ denote an $A$-symbol equivalence class
and let $S'\coloneqq S\setminus\{\max(S)\}$.  We ensure the symbols of $S$ appear in sorted
order in the first column of $A$ using the constraints
\begin{equation} \bigwedge_{\substack{s\in S'\\0\leq i'<i\leq9}}(A_{i,0,s}\limp\lnot A_{i',0,s+1}) . \end{equation}
Proposition~\ref{prop:min-b-sym} is handled in the same way.  Letting $T$ denote a $B$-symbol
equivalence class and $T'\coloneqq T\setminus\{\max(T)\}$, we use the constraints
\begin{equation} \bigwedge_{\substack{t\in T'\\0\leq i'<i\leq9}}(B_{i,0,t}\limp\lnot B_{i',0,t+1}) . \end{equation}

Finally, we discuss Proposition~\ref{prop:min-swap}, which applies
in all cases except case~2 and says that $A<B$.
Since we are not in case~2, we have $A_{0,0}=B_{0,0}$ by Proposition~\ref{prop:a00b00}.  As a result,
we start by considering the $(1,0)$th entries of $A$ and $B$, noting that $A<B$
implies that $A_{1,0}\leq B_{1,0}$.  We encode $A_{1,0}\leq B_{1,0}$ into Boolean logic
with the constraints
\begin{equation} \bigwedge_{\substack{0\leq k,l\leq9\\k>l}}(A_{1,0,k}\limp\lnot B_{1,0,l}) . \end{equation}
Next, we consider the case when $A_{1,0}=B_{1,0}$.  In this case, $A<B$ implies that
$A_{2,0}\leq B_{2,0}$, and we encode $A_{1,0}=B_{1,0}\limp A_{2,0}\leq B_{2,0}$
into Boolean logic with the constraints
\begin{equation} \bigwedge_{\substack{0\leq k,l,m\leq9\\k>l}}\bigl((A_{1,0,m}\land B_{1,0,m}\land A_{2,0,k})\limp\lnot B_{2,0,l}\bigr) . \end{equation}
We could continue in this fashion and also encode constraints corresponding to
$(A_{1,0}=B_{1,0}\land A_{2,0}=B_{2,0})\limp A_{3,0}\leq B_{3,0}$, etc.  However, in
practice it was sufficient to only consider the entries to the $(2,0)$th entry.  In other words,
we did not encode the full constraint $A<B$ in our SAT instances
but the strictly weaker constraint $[A_{0,0},A_{1,0},A_{2,0}]\leq[B_{0,0},B_{1,0},B_{2,0}]$.

\section{Exhaustive Enumeration and Proof Generation}\label{sec:enumerating}

This section explains the process by which we use a SAT solver to
find all solutions of a SAT instance (see Section~\ref{sec:exhaustive}) and the process
by which we generate and check proof certificates that the enumeration
was performed correctly, with no missing solutions (see Section~\ref{sec:proofs}).

\subsection{Exhaustive enumeration}\label{sec:exhaustive}

Typical modern SAT solvers stop as soon as a satisfying assignment is found and do not
support exhaustively enumerating all solutions of a SAT instance.
However, the IPASIR-UP interface~\cite{Fazekas2024}, as supported by the SAT solver
\textsc{CaDiCaL}~\cite{Biere2024}, enables us to find all solutions.
IPASIR-UP is an interface that can be used to inject custom code into a SAT solver
in order to change its behaviour.  One function supported by IPASIR-UP is
\verb|cb_check_found_model|, a function that is called when the SAT solver has found a new solution.
Inside \verb|cb_check_found_model| users are able to add code that modifies a SAT instance whenever
a solution is found.

In our case, we add a ``blocking clause'' into the SAT instance
every time a solution is found that prevents the solution from occurring again.  Once the blocking
clause has been added, the solver
continues looking for a new solution.  Eventually, once all solutions have been found, the solver
reports that the updated instance (i.e., the instance augmented with all blocking clauses)
is \emph{unsatisfiable}---it has no solutions.

The blocking clause that we inject into the SAT instance must only block the single
solution that was found and no others.  Suppose $(S,T)$ is the pair of orthogonal Latin squares
that the solver found.  We want to add the constraint
\begin{equation} \lnot\Bigl( \bigwedge_{0\leq i,j\leq9} (A_{i,j,S_{i,j}}\land B_{i,j,T_{i,j}}) \Bigr) , \end{equation}
which is logically equivalent to
$\bigvee_{0\leq i,j\leq9}(\lnot A_{i,j,S_{i,j}}\lor\lnot B_{i,j,T_{i,j}})$.
Note that this clause contains $2\cdot10^2=200$ literals.  An observation that allows us to shorten
this clause is to note that it is sufficient to block only the upper-left $9\times9$ entries
in $(S,T)$, because once those entries have been fixed the remaining entries are forced by the
Latin square constraints.  This allows us to replace the bound $0\leq i,j\leq9$ in the blocking clause
with the bound $0\leq i,j\leq8$, thereby shrinking the clause to $2\cdot9^2=162$ literals.

\subsection{Proof generation and verification}\label{sec:proofs}

All our calls to a SAT solver will eventually finish with an unsatisfiable result (i.e., no solutions)
as a result of the blocking clauses that we inject into the SAT instance in order to perform an exhaustive search.
Modern SAT solvers such as \textsc{CaDiCaL} support generating a ``proof certificate'' of unsatisfiability.
The proof certificate contains a log of the deductions that the solver
made in order to determine that the SAT instance has no solutions.  The certificate can then be checked by a proof
verifier, a separate program that verifies each deduction in the proof logically follows from the previous deductions.

The certificates we generate are based on the DRAT proof format~\cite{Wetzler2014}.
A DRAT certificate consists of the list of clauses deduced by the solver
in the order in which they were deduced.  The final clause in an unsatisfiability certificate
will be the empty clause.  The empty clause being a logical consequence of the original SAT instance
proves that the original instance was unsatisfiable, since no truth assignments satisfy the empty clause.

Every step in a DRAT proof is classified as either an \emph{addition}---a clause that can be deduced
from previously deduced clauses or clauses in the original SAT instance---or
a \emph{deletion}, a clause that was previously added but is no longer needed and should be removed.
Our work uses a simple extension of the DRAT format first proposed by Bright et al.~\cite{Bright2020}.  In this
extension a third kind of step is supported, a \emph{trusted addition}---a
clause that will be added into the list of clauses in the proof
even though it cannot necessarily be deduced from the previous clauses in the proof.

We require trusted clauses in our proofs
because the blocking clauses we used to perform exhaustive enumeration
were added through the IPASIR-UP interface, not deduced by the solver, and therefore
cannot be derived through the typical logical deduction process.
Thus, in our DRAT proofs when a blocking clause is generated a trusted addition is written into the DRAT proof.

A \emph{proof verifier} takes as input the SAT instance
and a DRAT proof of unsatisfiability and verifies every addition step in the proof logically follows from
the current set of derived clauses in conjunction with the clauses in the original SAT instance.
Deletion steps remove clauses in the current set of derived clauses when they are no longer needed in order to improve the efficiency of the
proof verifier.  Trusted addition steps add a clause into the current set of derived clauses without
verifying its deducibility.

\section{Results}\label{sec:results}

\begin{table}[b]
\caption{A summary of the running times (in seconds) of the 45 instances run in each of the five cases.
The minimum DRAT proof size is also provided as well as the number of solutions found in each case.}
\centering
\begin{tabular}{c@{\quad}r@{\quad}r@{\quad}r@{\quad}r@{\quad}c@{\quad}r}
case &        mean &      median &     minimum &     maximum & proof size & solutions \\
   1 &      5971.2 &      6116.8 &      4467.2 &      7307.3 & 3.6 GiB & 3904\phantom{00} \\
   2 &         0.9 &         0.9 &         0.7 &         1.4 & 2.1 MiB & 0\phantom{00} \\
   3 &         2.1 &         2.1 &         1.7 &         2.6 & 4.1 MiB & 0\phantom{00} \\
   4 &         2.0 &         1.9 &         1.5 &         2.8 & 4.0 MiB & 0\phantom{00} \\
   5 &      1981.2 &      1965.4 &      1775.2 &      2267.8 & 1.6 GiB & 22320\phantom{00}
\end{tabular}
\label{tbl:results}
\end{table}

We now discuss our computational results enumerating all 4-nets(10) with at least two
nontrivial relations.  Our results were run on an Intel i7 CPU running at 2.8~GHz
and using the SAT solver \textsc{CaDiCaL}~1.9.4 using up to~250 MiB of memory.
Our scripts are freely available and archived on Zenodo at \href{https://doi.org/10.5281/zenodo.17352786}{doi.org/10.5281/zenodo.17352786}.

We use a Python script to generate SAT instances in each of the five possible forms (cases 1--5)
following the encoding described in Section~\ref{sec:encoding}.  In order to mitigate the effect
of randomness in the search, each case was independently solved 45 times, each time with a different
random seed.  The differing random seeds ensure that each instance of \textsc{CaDiCaL} makes
different choices during the solving process.  A tabular summary of the results of these trials is
provided in Table~\ref{tbl:results}, and a box plot of the running times is provided in
Figure~\ref{fig:results}.

The SAT solver determined that cases 2, 3, and 4 all had no solutions and these cases were always solvable
in a few seconds.  It is interesting to note that Delisle~\cite{delisle2010search} ruled out cases
2 and 4 theoretically using a counting argument, but despite the fact we did not explicitly use this
in the SAT encoding, the SAT solver quickly proved unsatisfiability on its own.
The fact that the solver also quickly ruled out case~3 suggested
to us that case 3
was also resolvable using a counting argument, and we were successful in finding one (included
at the end of this section).
Unfortunately, the proofs produced by the SAT solver,
while logically correct, are not intended to be human-readable and the contents of the proofs did
not provide us with any mathematical insight.

In case 1 the SAT solver found 3,904 solutions, and in case 5 the SAT solver found 22,320 solutions.
The latter count agrees with the count reported by Delisle, but the former count is
exactly half of Delisle's count.  We contacted Delisle and Myrvold (who ran independent searches for
4-nets(10) with two relations) and their complete enumeration
in cases~1 and~5 matched ours exactly, so the count reported by Delisle in case~1
was simply a misprint.  We also verified that up to main class equivalence there are exactly 91 solutions
(7 in case~1 and 84 in case~5) and that 6 of these (all in case~5) are of rank 34
while the other 85 are of rank 35.

\begin{figure}[t]
    \centering
    \begin{subfigure}[b]{0.501\textwidth}
        \includegraphics[width=\textwidth]{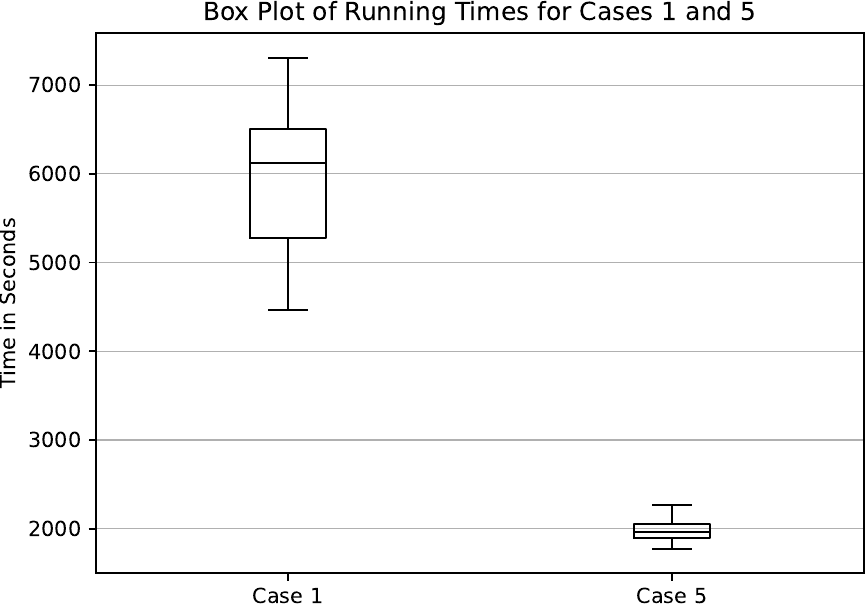}
        \caption{Cases 1 and 5}
    \end{subfigure}
    \begin{subfigure}[b]{0.490\textwidth}
        \includegraphics[width=\textwidth]{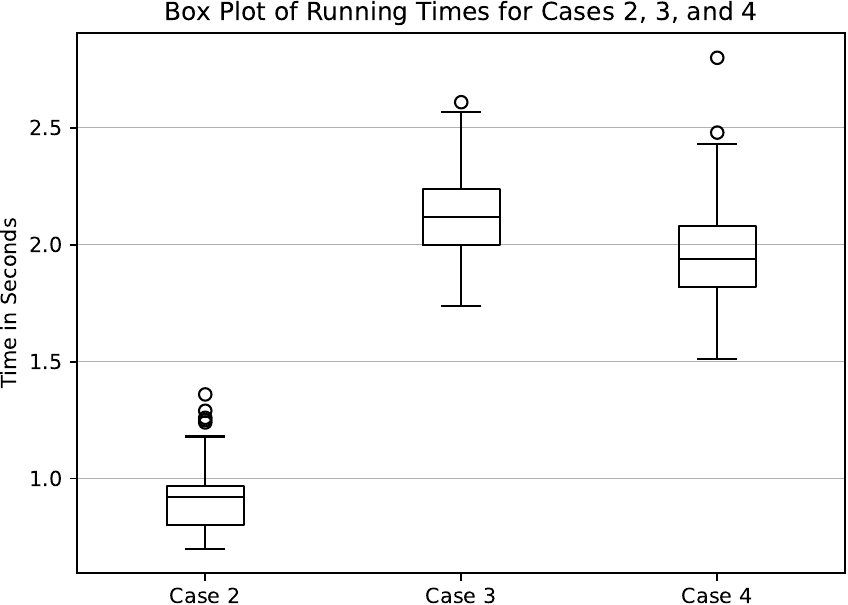}
        \caption{Cases 2, 3, and 4}
    \end{subfigure}

    \caption{A box plot visualization of the running times of the 45 instances solved in each case.
    }
    \label{fig:results}
\end{figure}

As mentioned in Section~\ref{sec:proofs}, \textsc{CaDiCaL} was configured to generate
DRAT proofs and each time a blocking clause was generated a ``trusted addition''
clause was added to the proof.  Since our proofs in
cases~1 and~5 use trusted additions for the blocking clauses, they cannot be verified using a standard DRAT proof checker like
\textsc{DRAT-trim}~\cite{Wetzler2014}.
However, \textsc{DRAT-trim-t} (a fork of \textsc{DRAT-trim}) supports trusted additions, so
we verified all our proofs using \textsc{DRAT-trim-t}~\cite{DRAT-trim-t}.  The proof size of the shortest proof produced in each case
is given in Table~\ref{tbl:results}.  The proofs for cases 2--4 were all verified in under a second,
while the proof in case~1 was verified in 3.2 hours and the proof in case~5 was verified in 0.9 hours.

The fact that the SAT solver was able to quickly rule out case~3 inspired us to look
for a counting argument that could rule out this case.  We were successful using an approach
similar to the arguments used by Gill and Wanless to count the number of points of certain
type in a net~\cite[Thm.~3.1]{Gill2023}.

\begin{theorem}\label{thm:pointtype}
There exist no 4-nets(10) with two relations in cases~2--4.
\end{theorem}

\begin{proof}
Say $R_1$ and $R_2$ are two relations in a 4-net(10).
In what follows we use the relational code `0' to denote $R_1\cap R_2$, `1' to denote $R_1\setminus R_2$,
`2' to denote $R_2\setminus R_1$, and `3' to denote $\overline{R_1}\cap\overline{R_2}$.
The \emph{type} of a point is a 4-character $\{0,1,2,3\}$-string denoting the point's relational codes
from each parallel class.  For example, a point of type 1122 is in $R_1$ but not $R_2$ in the first
two classes, and is in $R_2$ but not $R_1$ in the last two classes.
Let $t_{ijkl}$ denote the number of points in the net of type $ijkl$.
Note that for $R_2$ to be a relation it must be the case that $t_{ijkl}=0$ when $i+j+k+l\not\equiv0\pmod{2}$,
and similarly for $R_1$ to be a relation it must be the case that $t_{ijkl}=0$ when
$\lfloor i/2\rfloor+\lfloor j/2\rfloor+\lfloor k/2\rfloor+\lfloor l/2\rfloor\not\equiv0\pmod{2}$,
so there are $4^4/4=64$ nonzero $t_{ijkl}$ variables.

Let $[[x_1,y_1,z_1],[x_2,y_2,z_2],[x_3,y_3,z_3],[x_4,y_4,z_4]]$ denote the form of $R_1$ and~$R_2$
as defined in Section~\ref{sec:background}.
Now count the number of points of type $00\mathord{*}\mathord{*}$ where the symbols `$*$' are arbitrary.
Note there are exactly $x_1x_2$ points which lie on both a line with index in $[0,x_1)$ and a line with index in $[10,10+x_2)$,
so
\begin{equation} \sum_{0\leq k,l\leq 3} t_{00kl} = x_1x_2 . \end{equation}
Similar equations can be derived by counting points of other types.
For example, there are exactly $x_1x_3$ points of type $0\mathord{*}0\mathord{*}$, exactly $y_1y_4$
points of type $1\mathord{*}\mathord{*}1$, and exactly $z_1x_2$ points of type $20\mathord{*}\mathord{*}$,
resulting in the equations
\begin{equation} \sum_{0\leq j,l\leq 3} t_{0j0l} = x_1x_3, \quad \sum_{0\leq j,k\leq 3} t_{1jk1} = y_1y_4, \quad \sum_{0\leq k,l\leq 3} t_{20kl} = z_1x_2 . \end{equation}

In case~3, the linear system corresponding to the $3^2\binom{4}{2}=54$ ways of fixing
two entries in the point type to values in $\{0,1,2\}$ when converted into reduced row echelon form
has a row corresponding to $t_{0123}-t_{3210}=-1/2$ which has
no integer solutions.

In cases~2 and~4, the linear systems corresponding to the $4^2\binom{4}{2}=96$ ways of fixing
two entries in the point type to values in $\{0,1,2,3\}$ are both inconsistent over the reals. \qed
\end{proof}
The counting argument in the proof of Theorem~\ref{thm:pointtype} provides some theoretical
conditions speculated on by Delisle in their original work:

\begin{quote}
\emph{Interestingly, no pairs of MOLS are completable for case 3. Some theoretical conditions
possibly exist to explain this, but are not known at this time.}~\cite{delisle2010search}
\end{quote}

\section{Conclusion}\label{sec:conclusion}

In this paper we recreated Delisle's 2010 enumeration of all
4-nets(10) with two nontrivial relations.  In contrast to Delisle's original search
that used a custom-written backtracking program, we use a SAT solver and found that
the SAT solver could complete the search over 6000 times faster (when run on modern hardware)
than the original backtracking code.  This is in part due to improvements in processing power,
but it is also due to the powerful search-with-learning algorithms used in modern SAT solvers
that can be effective at solving problems in design theory.
For example, the author of the SAT solver SATO, H.~Zhang, observed SAT solvers are particularly effective
at solving Latin square problems:
\begin{quote}
\emph{In the earlier stage of our study of Latin square problems,
the author wrote two special-purpose programs.
After observing that these two programs could not do better than SATO,
the author has not written any special-purpose search programs since then.}~\cite{Zhang2021}
\end{quote}
Moreover, our results are more trustworthy in the sense that they do not require trusting
the implementation of a search algorithm.  Instead, we generate certificates that our search was
exhaustive, and our results only require trusting the reduction of
the problem into Boolean logic (as described in Section~\ref{sec:encoding})
and the proof verifier that we use (as described in Section~\ref{sec:proofs}).

For future work, it would be interesting to use a SAT solver to investigate
the results of Gill and Wanless~\cite{Gill2023} who enumerated all 4-nets(10)
with a single nontrivial relation and ruled out the existence of a relation
of type $[2,2,2,4,6]$ in a 5-net(10).  More potential future work would be to rule
out the existence of other relation types in a 5-net(10) or a 6-net(10),
or to rule out the existence of two nontrivial relations in a 5-net(10) or a
6-net(10).  If the latter was accomplished, the work of Howard~\cite{howardthesis}
would imply that $\MOLS{4}{10}$ do not exist.  SAT solvers might also be useful
in exploring the existence of other combinatorial designs whose existence is uncertain,
such as non-Desarguesian projective planes of order~11.

\paragraph{Acknowledgements} We would like to thank Wendy Myrvold and Erin Delisle
who answered questions about their work and provided the 26,224 pairs of Latin squares
found in their search.  We would also like to thank the anonymous reviewer who
provided useful feedback on this manuscript.

\bibliographystyle{splncs04}
\bibliography{bibliography}

\end{document}